\documentclass[11pt,reqno]{amsart}

\usepackage{amscd,amssymb,amsmath,amsthm}
\usepackage{graphicx}
\usepackage{color}
\usepackage{cite}
\newtheorem{thm}{Theorem}
\newtheorem{defn}{Definition}
\newtheorem{lemma}{Lemma}
\newtheorem{pro}{Proposition}
\newtheorem{rk}{Remark}
\newtheorem{cor}{Corollary}

\numberwithin{equation}{section} 

\makeatletter
\newcommand*{\rom}[1]{\expandafter\@slowromancap\romannumeral #1@}
\makeatother

\begin{document}
\title[On Gibbs measure of HC model]{A HC model with countable set of spin values: uncountable set of Gibbs measures}

\author{U.~A.~Rozikov, F.~H.~Haydarov}

\address{U.~A.~Rozikov, F.~H.~Haydarov}

\address{Institute of mathematics,
9, University str. Tashkent, 100174, Uzbekistan,}

\address{AKFA University, 264, Milliy Bog street,  Yangiobod QFY, Barkamol MFY,
      Kibray district, 111221, Tashkent region, Uzbekistan,}
      
\address{National University of Uzbekistan,
 University str 4 Olmazor district, Tashkent, 100174, Uzbekistan.}

\email {rozikovu@mail.ru, haydarov\_imc@mail.ru}

\begin{abstract} We consider a hard core (HC) model with a countable set  $\mathbb{Z}$ of spin values on the Cayley tree. This model is defined by a countable set of parameters $\lambda_{i}>0, i \in \mathbb{Z}\setminus\{0\}$. For all possible values of parameters, we give limit points of the dynamical system generated by a function which describes the consistency condition for finite-dimensional measures. Also, we prove that every periodic Gibbs measure for the given model is either translation-invariant or periodic with period two. Moreover, we construct uncountable set of Gibbs measures for this HC model.
\end{abstract}
\maketitle

{\bf Mathematics Subject Classifications (2010).} 82B05, 82B20
(primary); 60K35 (secondary)

{\bf{Key words.}} Cayley tree, Gibbs measure, HC model, dynamics, Bleher-Ganikhodjaev construction.

\section{Introduction}




In the theory of Gibbs measures important problems are related to
existence and non-uniqueness of such measures.
The existence is known for models with a finite number of spin values,
but it depends on Hamiltonian of the system in case of a countable set of spin values (see \cite{georgii, 2}).


 There are papers devoted to the study of (gradient) Gibbs measures for models with an infinite set of spin values. For gradient Gibbs measures of gradient potentials on Cayley tree see \cite{9,  a, b, c} and references therein. In \cite{5} the uniqueness of the translation-invariant Gibbs measure for the antiferromagnetic Potts model with a countable set of spin values and a nonzero external field was shown. In  \cite{6}  for this Potts model the  Poisson measures, which are Gibbs measures, were described.

 HC models arise in statistical mechanics, combinatorics, and neural networks \cite{12, 13}. In \cite{14}, Mazel and Suhov introduced and studied the HC model on a $d$-dimensional lattice $\mathbb{Z}^{d}$. Many papers are devoted to the study of Gibbs measures for HC models with a finite number of states on the Cayley tree (see \cite{12}, \cite{16}, \cite{13}, \cite{15}, \cite{5} and the references therein).

In \cite{24}, for the first time a HC model with a countable number of spin values is considered. This model is defined by a countable set of parameters $\left\{\lambda_{j}\right\}_{j \in \mathbb{Z}}$. In the paper, the exact value of the parameter is found, which is the sum of the series obtained from the sequence of parameters $\left\{\lambda_{j}\right\}_{j \in \mathbb{Z}}$ such that for $0<\Lambda \leq \Lambda_{c r}$ there is exactly one periodic Gibbs measure $\mu_{0}$, which is translation invariant, and for $\Lambda>\Lambda_{c r}$ there are exactly three periodic measures Gibbs $\mu_{0}, \mu_{1}, \mu_{2}$, where the measures $\mu_{1}$ and $\mu_{2}$ are periodic (not translation-invariant) Gibbs measures with period two.

In the present paper, we continue investigation of the model considered in \cite{24}. For all possible parameters $\left\{\lambda_{j}\right\}_{j \in \mathbb{Z}}$, we give all limit points of the dynamical system generated by a function which describes the consistency condition for finite-dimensional Gibbs distributions. It is known (see \cite{Zach}) that there is one-to-one correspondence between   Gibbs measure and normalisable boundary laws. For our model we show that any boundary law is normalisable. This result then applied to prove that every periodic Gibbs measure for the given model is either translation-invariant or periodic with period two, since such measures are completely studied in \cite{24} our result says that there is no any other periodic Gibbs measure. In addition, by adapting to our model the (well known for Ising model) Bleher-Ganikhodjaev construction (see \cite{BG}) we construct an uncountable set of Gibbs measures which correspond to an uncountable set of normalisable boundary laws.

\section{Preliminaries}

The Cayley tree $\Im^{k}=(V, L)$ of order $k \geq 1$ is an infinite tree, i.e. graph without cycles, each vertex of which has exactly $k+1$ edges. Here $V$ is the set of vertices of $\Im^{k}$ va $L$ is the set of its edges. For $l \in L$ its endpoints $x, y \in V$ are called nearest neighbor vertices and denoted by $l=\langle x, y\rangle$.

For a fixed $x^{0} \in V$ we put
$$
W_{n}=\left\{x \in V \mid d\left(x, x^{0}\right)=n\right\}, \quad V_{n}=\bigcup_{m=0} ^{n} W_{m}, $$
%
where $d(x, y)$ is the distance between the vertices $x$ and $y$ on the Cayley tree, i.e. number of edges of the shortest path connecting vertices $x$ and $y$.


 The set $S(x)$ of direct successors of the vertex $x$ is defined as follows. If $x \in W_{n}$ then
$$
S(x)=\left\{y_{i} \in W_{n+1} \mid d\left(x, y_{i}\right)=1, i=1,2, \ldots, k \right\}.
$$

Let us consider the HC model of nearest neighbors with a countable number of states on the Cayley tree. The configuration $\sigma=\{\sigma(x) \mid x \in V\}$ on the Cayley tree is given as a function from $V$ to the set $\mathbb{Z}$, i.e. in this model, each vertex $x$ is assigned one of the values $\sigma(x) \in \mathbb{Z}$, where $\mathbb{Z}$ is the set of integer numbers.


Consider the set $\mathbb{Z}$ as the set of vertices of some infinite graph $G$. Using the graph $G$, we define a $G$-admissible configuration as follows: a configuration $\sigma$ is called a $G$-admissible configuration on a Cayley tree if $\{\sigma(x), \sigma(y)\}$  is an edge of the graph $G$ for any nearest neighbors $x, y$ from $V$.

Denote the set of $G$-admissible configurations by $\Omega^{G}$.

The activity set \cite{16} for the graph $G$ is the bounded function $\lambda: G \mapsto \mathbb{R}_{+}$ ($\mathbb{R}_{+}$ is the set of positive real numbers). The value $\lambda_{i}$ of the function $\lambda$ at the vertex $i \in \mathbb{Z}$ is called its ``activity".

For given $G$ and $\lambda$, we define the $G-$ HC-model Hamiltonian as
\begin{equation}\label{1}
H_{G}^{\lambda}(\sigma)= \begin{cases}J \sum\limits_{x \in V} \ln \lambda_{\sigma(x)}, & \text { if } \sigma \in \Omega^{G}, \\ +\infty, & \text { if } \sigma \notin \Omega^{G},\end{cases}
\end{equation}
where $J \in \mathbb{R}$.

Let $L(G)$ be the set of edges of the graph $G$, denote $A \equiv A^{G}=\left(a_{i j}\right)_{i, j \in \mathbb Z}$ by the adjacency matrix of $G$, i.e.
$$
a_{i j}=a_{i j}^{G}=\left\{\begin{array}{lll}
1 & \text { if } & \{i, j\} \in L(G), \\
0 & \text { if } & \{i, j\} \notin L(G).
\end{array}\right.
$$

%

%
%

\begin{defn}\label{Definition 3} (see chapter 12 of \cite{georgii}) A family of vectors $l=\left\{l_{xy}\right\}_{\langle x, y) \in L}$ with $l_ {xy}=$ $\left\{l_{xy}(i): i \in \mathbb{Z}\right\} \in(0, \infty)^{\mathbb{Z}}$ is called the boundary law for the Hamiltonian (\ref{1}) if

1) for each $\langle x, y\rangle \in L$ there exists a constant $c_{x y}>0$ such that the consistency equation
\begin{equation}\label{5}
l_{xy}(i)=c_{xy} \lambda_{i} \prod_{z \in \partial x \backslash\{y\}} \sum_{j \in \mathbb{Z}} a_{ij} l_{zx}(j)
\end{equation}
holds for any $i \in \mathbb{Z}$, where $\partial x$ is the set of nearest neighbors of $x$.

2) The boundary law $l$ is said to be normalisable if and only if
\begin{equation}\label{6}
\sum_{i \in \mathbb{Z}}\left(\lambda_{i} \prod_{z \in \partial x} \sum_{j \in \mathbb{Z}} a_{ij} l_{zx} (j)\right)<\infty
\end{equation}
for all $x \in V$.
\end{defn}
Let us now give a general setting of correspondence between boundary laws and Gibbs measures.  To do this consider a nearest-neighboring interaction potential $\Phi=(\Phi_b)_b$, where
$b=\langle x,y \rangle$ is an edge,  define symmetric transfer matrices $Q_b$ by
\begin{equation}\label{Qd}
	Q_b(\omega_b) = e^{- \big(\Phi_b(\omega_b) + | \partial x|^{-1} \Phi_{\{x\}}(\omega_x) + |\partial y |^{-1} \Phi_{\{y\}} (\omega_y) \big)},
\end{equation}
where $\omega_b=\{\omega_x,\omega_y\}$, $\partial x$ is the set of all nearest-neighbors of $x$ and $|S|$ denotes the number of elements of the set $S$.

For a finite subset $\Lambda\subset V$ define the (Markov) Gibbsian specification as
$$
\gamma_\Lambda^\Phi(\sigma_\Lambda = \omega_\Lambda | \omega) = (Z_\Lambda^\Phi)(\omega)^{-1} \prod_{b \cap \Lambda \neq \emptyset} Q_b(\omega_b).
$$

\begin{thm}\label{tz} \cite{Zach}
	For any Gibbsian specification $\gamma$ with associated family of transfer matrices $(Q_b)_{b \in L}$  we have
	\begin{enumerate}
		\item Each {\it normalisable} boundary law $(l_{xy})_{x,y}$ for $(Q_b)_{b \in L}$ defines a unique Gibbs measure\footnote{Also is called tree-indexed Markov chain.} $\mu$ (corresponding to $\gamma$)  via the equation given for any connected set $\Lambda \subset V$
		\begin{equation}\label{BoundMC}
			\mu(\sigma_{\Lambda \cup \partial \Lambda}=\omega_{\Lambda \cup \partial \Lambda}) = (Z_\Lambda)^{-1} \prod_{y \in \partial \Lambda} l_{y y_\Lambda}(\omega_y) \prod_{b \cap \Lambda \neq \emptyset} Q_b(\omega_b),
		\end{equation}
		where for any $y \in \partial \Lambda$, $y_\Lambda$ denotes the unique nearest-neighbor of $y$ in $\Lambda$.
		\item Conversely, every Gibbs measure $\mu$ admits a representation of the form (\ref{BoundMC}) in terms of a {\it normalisable} boundary law (unique up to a constant positive factor).
	\end{enumerate}
\end{thm}

Assuming $l_{x y}(0) \equiv 1$ (normalization at 0 ), from (\ref{5}) we obtain
\begin{equation}\label{7}
l_{xy}(i)=\frac{\lambda_{i}}{\lambda_{0}} \prod_{z \in \partial x \backslash\{y\}} \frac{a_{i 0}+ \sum_{j \in \mathbb{Z}_{0}} a_{ij} l_{zx}(j)}{a_{00}+\sum_{j \in \mathbb{Z}_{0}} a_{0 j} l_{zx}(j)}.
\end{equation}

In this paper we consider a {\bf concrete graph} $G^{\ast}$ with $a_{i 0}=1$ for all $i \in \mathbb{Z}$ and $a_{im}=0$ for all $i, m \in \mathbb{Z}_ {0}$.\\

The corresponding admissible configuration satisfies $\sigma(x) \sigma(y)=0$ for any nearest neighbor $\langle x, y\rangle$ from $V$, i.e. if a vertex $x$ has a spin value $\sigma(x)=0$, then we can put any value from $\mathbb{Z}$ on neighboring vertices, if the vertex $x$ contains any value from $\mathbb{Z}_{0}$, then we put only zeros on the neighboring vertices.

Given a boundary law $l_{xy}(i)$, we define $z_{i,x}=l_{xy}(i)$ when $x$ is  direct successor of $y$, i.e. $x\in S(y)$, then (\ref{7})  can be written as (see \cite{18})

\begin{equation}\label{8}
z_{i, x}=\lambda_{i} \prod_{y \in S(x)} \frac{1}{1+\sum_{j \in \mathbb{Z}_{0}} z_{j , y}}, \quad i \in \mathbb{Z}_{0}.
\end{equation}
Thus the investigation of the Gibbs measures for Hamiltonian (\ref{1}) in case of graph $G^{\ast}$ is reduced to finding solutions of (\ref{8}). In the following sections we give several solutions of (\ref{8}).

\section{A dynamical system corresponding to (\ref{8})}
Denote
$$\mathcal{V}:=\left\{z_x=\left(\ldots z_{-2,x}, z_{-1,x}, z_{0,x}, z_{1,x}, z_{2, x}, \ldots\right) \mid z_{i,x} \in \mathbb{R}, \ z_{0,x}=1, \ \forall x\in V \right\}, $$
$$\mathcal{V}^+:=\left\{z_x\in \mathcal{V} \mid z_{i,x}>0\right\},$$
$$\mathcal{L}:=\left\{z_x\in \mathcal{V} \mid \sum_{i \in \mathbb{Z}}\left|z_{i, x}\right|<\infty\right\}, \quad \mathcal{L}^+:=\mathcal{V}^+\cap\mathcal{L}.$$

Hence, (\ref{8}) can be written as
\begin{equation}\label{9}
	z_x=\lambda \prod_{y\in S(x)}\frac{1}{1+\|z_y\|}, \quad z_x\in \mathcal{V}^+, \ \forall x\in V, \end{equation}
where $\lambda=(..., \lambda_{-2}, \lambda_{-1}, 1, \lambda_1, \lambda_2, ...)\in \mathcal{V}^+$ and $\|z_y\|:=\sum_{i\in\mathbb{Z}}\|z_{i,y}\|.$

Denote
$$S_{1}(x)=\left\{y \in V:\langle x, y\rangle\right\}, x \in V \ \textrm{and} \ x_{\downarrow}=S_1(x)\setminus S(x).$$

\begin{lemma}\label{n1}
	$\lambda \in \mathcal{L}^+$ if and only if $z_x\in \mathcal{L}^+$.
\end{lemma}
\begin{proof}
If $\|\lambda\|=\infty$ then by $\left\|z_{x}\right\| \prod_{y \in S(x)}\left(1+\left\|z_{y}\right\|\right)=\|\lambda\|$,
there exists $v \in S(x) \cup\{x\}$ such that $\left\|z_{v}\right\|=\infty .$
Thus, $z_{\downarrow v}=\textbf{0}\notin \mathcal{L}^+$, where $\textbf{0}$ is a zero vector. Hence, it is sufficient to consider the case $\lambda \in \mathcal{L}^+$. Since $$\left\|z_{x}\right\|=\|\lambda\| \cdot \prod_{y \in S(x)} \frac{1}{1+\left\|z_{y}\right\|} \leqslant\|\lambda\|$$ we get $z_{x} \in \mathcal{L}^+$ for all $x \in V$.
\end{proof}

We define an operator $W: \mathcal{L}^+ \rightarrow\langle\lambda\rangle:=\{c\lambda: c\in \mathbb{R}\}$ by
\begin{equation}\label{operator}W(z)=\lambda \cdot \frac{1}{(1+\|z\|)^{k}}, \ \lambda \in \mathcal{L}^+.\end{equation}
For each $z_{0} \in \mathcal{L}^{+}$, we study $\lim_{n \rightarrow \infty} W^{(n)}\left(z_{0}\right)$, where
$$W^{(n)}=\underbrace{W\circ W\circ ... \circ W}_{n \ \textrm{times}}.$$\\
Denote $$W^{(n)}\left(z_{0}\right)=z^{(n)}=\alpha_{n} \lambda \in\langle\lambda\rangle.$$
Note that
$$z_0=\alpha_0 \lambda.$$
From
$$\alpha_{n+1}\lambda=z_{n+1}=W^{(n+1)}\left(z_{0}\right)
=W\left(z_{n}\right)=W\left(\alpha_{n} \lambda\right)$$ we obtain $$\alpha_{n+1}\lambda=W\left(\alpha_{n} \lambda\right)
=\lambda \cdot \frac{1}{\left(1+\left\|\lambda \alpha_{n}\right\|\right)^{k}}=\lambda \cdot \frac{1}{\left(1+\alpha_{n}\|\lambda\|\right)^{k}}.
$$
Consequently,
\begin{equation}\label{bob1}\quad \alpha_{n+1}=\frac{1}{\left(1+\alpha_{n}\|\lambda\|\right)^{k}}=: f_{\lambda}\left(\alpha_{n}\right).\end{equation}
Hence, we now study limit of $\alpha_{n}$.
It's easy to check that
$f_{\lambda}\left(x\right)$ is a monotone decreasing function. Then
\begin{equation}\label{bob2} g_{\lambda}\left(x\right):=(f_{\lambda}\circ f_{\lambda})\left(x\right)\end{equation} is a monotone increasing function.
If $\alpha_0\leq \alpha_2$ (the case $\alpha_0\geq \alpha_2$ is similar) then
$g_{\lambda}\left(\alpha_0\right)\leq g_{\lambda}\left(\alpha_2\right)\ \Rightarrow \ \alpha_2\leq \alpha_4.$
Continuing this process we obtain that $\{\alpha_{2n}\}_{n=1}^{\infty}$ is the monotone increasing sequence ($\{\alpha_{2n-1}\}_{n=1}^{\infty}$ is the monotone decreasing sequence). Taking into account both cases we can say $\{\alpha_{2n}\}_{n=1}^{\infty}$ is the monotone sequence.  Similarly, we can show $\{\alpha_{2n-1}\}_{n=1}^{\infty}$ is the monotone sequence. It is easy to check that $0<\alpha_{n}<1$, i.e. $\left\{\alpha_{n}\right\}_{n=1}^{\infty}$ is the bounded sequence.
Hence, there exist limits of sequences $\{\alpha_{2n-1}\}_{n=1}^{\infty}$ and $\{\alpha_{2n}\}_{n=1}^{\infty}$. Put $$\lim _{n \rightarrow \infty} \alpha_{2 n-1}=A, \ \lim _{n \rightarrow \infty} \alpha_{2 n}=B.$$
Since $\alpha_{n+1}=\frac{1}{\left(1+\alpha_{n} \cdot\|\lambda\|\right)^{k}}, \ n\in \mathbb{N}$ one gets
\begin{equation}\label{eq1} A=f_{\lambda}(B), \ B=f_{\lambda}(A). \end{equation}

We note that if $A\neq B$ and $\|\lambda\|\leq \Lambda_{c r}(k)=\frac{k^{k}}{(k-1)^{k+1}}$ then the last equation has not any positive solution (see \cite{24}). For the case $A=B$ and $\|\lambda\|\leq\Lambda_{c r}(k)$ there is a unique positive solution $\xi$. If $A\neq B$ and $\|\lambda\|>\Lambda_{c r}(k)$ then equation (\ref{eq1}) has two positive solutions $(\alpha^{\ast}, \beta^{\ast})$ and $(\beta^{\ast}, \alpha^{\ast})$, $\alpha^{\ast}<\beta^{\ast}$. For the case $A=B$ and $\|\lambda\|>\Lambda_{c r}(k)$ there is a unique positive solution $\xi$. After simple analyzing, the function $g_{\lambda}(x)$ is convex on $(0, \alpha^{\ast}]\cup [\xi, 1]$ and concave on $(\alpha^{\ast}, \xi)$. Also, if $z_0=\alpha_0\lambda$, with $\alpha_0\in (0, \xi)$ then $\lim_{n\rightarrow \infty}W^{(2n)}(z_0)=\alpha^{\ast}$ and $\lim_{n\rightarrow \infty}W^{(2n-1)}(z_0)=\beta^{\ast}$. For the case $\alpha_0=\xi$ we have $\lim_{n\rightarrow \infty}W^{(n)}(z_0)=\xi$. If $\alpha_0\in (\xi, 1]$ then $\lim_{n\rightarrow \infty}W^{(2n)}(z_0)=\beta^{\ast}$ and $\lim_{n\rightarrow \infty}W^{(2n-1)}(z_0)=\alpha^{\ast}$.\\

All of the above mentioned, we can conclude that:

\begin{thm}\label{thm4} Let $k\geq 2$ and $\lambda\in \mathcal{L}^+$.
	\begin{enumerate}
		\item If $\|\lambda\|\leq \Lambda_{cr}$ then $\lim_{n\rightarrow \infty} W^{(n)}(z_0)=\xi \lambda$, for all $z_0\in \mathcal{L}^+$.
		\item If $\|\lambda\|> \Lambda_{cr}$ and $z_0=\alpha_0\lambda$, with $\alpha_0\in (0,\xi)$ then $$\lim_{n\rightarrow \infty} W^{(2n)}(z_0)=\alpha^{\ast} \lambda, \ \lim_{n\rightarrow \infty} W^{(2n-1)}(z_0)=\beta^{\ast} \lambda.$$
		\item If $\|\lambda\|> \Lambda_{cr}$ and $z_0=\alpha_0\lambda$, with $\alpha_0\in (\xi, 1)$ then $$\lim_{n\rightarrow \infty} W^{(2n)}(z_0)=\beta^{\ast} \lambda, \ \lim_{n\rightarrow \infty} W^{(2n-1)}(z_0)=\alpha^{\ast} \lambda.$$
		\item If $\alpha_0=\xi$ then $\lim_{n\rightarrow \infty} W^{(n)}(z_0)=\xi \lambda$.
	\end{enumerate}
\end{thm}
\section{Each solution of (\ref{8}) is normalisable}
From above sections, all periodic solutions to the operator $W$ which defined in (\ref{operator}) is either translation invariant or periodic with period two. In addition, if $\|\lambda\|>\Lambda_{cr}$ then there is only one periodic solution with period two $(\alpha^{\ast}\lambda, \beta^{\ast}\lambda)$ such that $\alpha^{\ast}<\beta^{\ast}$, i.e.
$$W(\alpha^{\ast} \lambda)=\beta^{\ast} \lambda, \quad W(\beta^{\ast} \lambda)=\alpha^{\ast} \lambda.$$

Let $u=(...,u_{-2},u_{-1}, 0, u_1, u_2, ...),v=(...,v_{-2}, v_{-1}, 0, v_1, v_2, ...) \in \mathcal{L}^{+}$. We introduce partial ordered relation $\preceq$ on the set $\mathcal{L}^{+}$ such that $u\preceq v$ if $u_i\leq v_i$ for all $i\in \mathbb{Z}$.

\begin{pro}\label{Proposition 2.5.}  If $z_{x}$ is a solution to (\ref{9}) then
	$\alpha^{\ast} \lambda \preceq z_{x} \preceq \beta^{\ast} \lambda$, for any $x \in V$.\end{pro}

\begin{proof} Since $W(z_x)\in \langle \lambda \rangle$ we can write $z_x=\alpha_x\lambda, \ \alpha_x>0$. Then the equation (\ref{9}) can be written as
	$$\alpha_x\lambda=\lambda\prod_{y\in S(x)}\frac{1}{1+\alpha_y\|\lambda\|}.$$
	Consequently, we obtain
	$$\alpha_x=\prod_{y\in S(x)}\frac{1}{1+\alpha_y\|\lambda\|}\ \Rightarrow \ 0<\alpha_x<1.$$
	Since $f_{\lambda}(x)$ (defined in (\ref{bob1})) decreasing function, one gets
	$$\lambda\alpha_{\lambda}:=\frac{1}{(1+\|\lambda\|)^k}\lambda=f_{\lambda}(1)\lambda\preceq z_x\preceq \lambda f_{\lambda}(0)=\lambda, \ x\in V\setminus \{x^{0}\}.$$
	Then
	$$f_{\lambda}(\lambda)\lambda \preceq z_x\preceq f_{\lambda}(\lambda\alpha_{\lambda})\lambda.$$
	Iterating this argument we obtain
	$$
	\lambda h^{(2n)}(\alpha_{\lambda})\preceq z_x \preceq\lambda h^{(2n)}(1), \quad \lambda h^{(2n-1)}(1)\preceq z_x\preceq\lambda h^{(2n-1)}(\alpha_{\lambda}),
	$$
	where $h^{(n)}$ is its $n$th iteration of the mapping $f_{\lambda}(\lambda)\lambda$. By Theorem \ref{thm4} we have
	$$\lim_{n\rightarrow \infty}h^{(2n)}(1)=\lim_{n\rightarrow \infty}h^{(2n-1)}(\alpha_{\lambda})=\beta^{\ast}\lambda$$
	and
	$$\lim_{n\rightarrow \infty}h^{(2n-1)}(1)=\lim_{n\rightarrow \infty}h^{(2n)}(\alpha_{\lambda})=\alpha^{\ast}\lambda$$
	Hence, we obtain that $\alpha^{\ast} \lambda \preceq z_{x} \preceq \beta^{\ast} \lambda$, for any $x \in V$. \end{proof}

\begin{pro}\label{lemma new} Let $k\geq 2$ and $\lambda\in \mathcal{L}^+$. Then any boundary law $l=\left\{l_{xy}:=z_{x}\right\}_{\langle x, y\rangle \in L}$, for the Hamiltonian (\ref{1}) with graph $G^*$, is normalisable.\end{pro}

\begin{proof} By Lemma \ref{n1} we obtain $z_x\in \mathcal{L}^+$. Since
	$$l_{xy}(i)=c_{xy}\prod_{z\in\partial x\setminus \{y\}}\sum_{j\in\mathbb{Z}}Q(i,j)l_{zx}(j)$$
	one gets
	$$\sum_{i\in \mathbb{Z}}\left(\prod_{z\in \partial x}\sum_{j\in \mathbb{Z}}Q(i,j)l_{zx}(j)\right)=\frac{1}{c_{xy}}
	\sum_{i\in\mathbb{Z}}\sum_{j\in\mathbb{Z}}Q(i,j)l_{yx}(j)l_{xy}(i),$$
	where $Q(i,j):=a_{ij}\lambda_i$. Hence, it is sufficient to show that
	$$\sum_{i\in\mathbb{Z}}\sum_{j\in\mathbb{Z}}z_{y, j}Q(i,j)z_{x, i}<\infty.$$
	Indeed, for graph $G^*$ we have
	$$\sum_{i\in\mathbb{Z}}\sum_{j\in\mathbb{Z}}z_{y, j}Q(i,j)z_{x, i}=\sum_{i\in\mathbb{Z}}\sum_{j\in\mathbb{Z}}a_{ij}\lambda_{j}z_{x,i}z_{y,j}=z_{y,0}\sum_{i\in\mathbb{Z}} z_{x,i}\lambda_i.$$
	By Proposition \ref{Proposition 2.5.} we obtain
	$$\sum_{i\in\mathbb{Z}}\sum_{j\in\mathbb{Z}}z_{y, j}Q(i,j)z_{x, i}\leq \beta^{\ast} z_{y,0}\sum_{i\in\mathbb{Z}} \lambda^{2}_i.$$
	It is easy to check that if $\lambda\in \mathcal{L}^+$ then $\sum_{i\in\mathbb{Z}}\lambda_i^2<\infty$.
\end{proof}

\begin{rk} In the sequel of the paper we assume $k\geq 2$ and $\lambda\in \mathcal{L}^+$. Therefore, each boundary law which we construct, by Theorem \ref{tz} and Proposition \ref{lemma new}, defines a Gibbs measure.
\end{rk}

\section{Periodic HC Gibbs measures}

{\it A group representation of the Cayley tree.} \,\ Let $G_{k}$
be a free product of $k+1$ cyclic groups of the second order with
generators $a_{1},a_{2},...a_{k+1},$ respectively. It is known that there exists a one-to-one correspondence between
the set of vertices $V$ of the Cayley tree $\Gamma^{k}$ and elements of the
group $G_{k}$. To give this correspondence we fix an arbitrary element $x_{0}\in
V$ and let it correspond to the unit element $e$ of the group
$G_{k}.$ Using $a_{1},...,a_{k+1}$ we numerate the
nearest-neighbors of element $e$, moving by positive direction.
Now we'll give numeration of the nearest-neighbors of each $a_{i},
i=1,...,k+1$ by $a_{i}a_{j}, j=1,...,k+1$. Since all $a_{i}$ have
the common neighbor $e$ we give to it $a_{i}a_{i}=a_{i}^{2}=e.$
Other neighbor are numerated starting from $a_{i}a_{i}$ by the
positive direction. We numerate the set of all the
nearest-neighbors of each $a_{i}a_{j}$ by words $a_{i}a_{j}a_{q},
q=1,...,k+1,$ starting from $a_{i}a_{j}a_{j}=a_{i}$ by the
positive direction. Iterating this argument one gets a one-to-one
correspondence between the set of vertices $V$ of the Cayley tree
and the group $G_{k}$ (e.g., \cite{5}).\medskip

\begin{defn}\label{defn2.1} Let $K$ be a subgroup of $G_k,$ $k\geq
1.$ We say that a function $h=\{h_x\in \mathbb R: x\in G_k\}$ is $K$-
periodic if $h_{yx}=h_x$ for all $x\in G_k$ and $y\in K.$ A $G_k$-
periodic function $h$ is called translation-invariant.
\end{defn}
A Gibbs measure is called $K$-\emph{periodic} if it corresponds to
$K$-periodic function $h$. Let $G_k: K=\{K_1,..., K_r\}$ be a
family of cosets, $K$ is a subgroup of index $r\in \mathbb{N}.$
\begin{defn}\label{defn 2.3} A set of quantities $h=\{h_x, x\in
G_k\}$ is called $K$-weakly periodic, if $h_x=h_{ij},$ for any
$x\in K_i, \ x_{\downarrow}\in K_j$.
\end{defn}

Let $K$ be a subgroup of index $r$ in $G_{k}$, and let $G_{k} / K=\left\{K_{0}, K_{1}, \ldots, K_{r-1}\right\}$ be the quotient group, with the $\operatorname{coset} K_{0}=K$.

 Put $$q_{i}(x)=\left|S_{1}(x) \cap K_{i}\right|, \quad i=0,1, \ldots, r-1,$$ where  $|\cdot|$ is the number of elements in the set.
  Also, we denote $$Q(x)=\left(q_{0}(x), q_{1}(x), \ldots, q_{r-1}(x)\right).$$

We note (see \cite{Rt} and Theorem 1.13 in \cite{5}) that for every $x \in G_{k}$ there is a permutation $\pi_{x}$ of the coordinates of the vector $Q(e)$ (where $e$ is the identity of $G_{k}$ ) such that
\begin{equation}\label{rb}
\pi_{x} Q(e)=Q(x).
\end{equation}

Let $G_{k}^{(2)}$ denote subgroup of $G_{k}$ consisting all words of even length. This is a normal subgroup of index two.

\begin{thm}\label{th1} Let $H$ be a normal subgroup of finite index in $G_{k}$. Then each $H$-periodic solution of (\ref{9}) for the HC-model with countable set of spin values is either translation-invariant or $G_{k}^{(2)}$- periodic (i.e. with period two). \end{thm}

\begin{proof} Let $H\triangleleft \ G_{k}$ and $G_{k}: H=\left\{H_{0}, H_{1}, H_{2}, \ldots, H_{n}\right\}$, where $H_0=H$. If $h_{1}, h_{2} \in H_{0}$ then $\left|\left\{h_{j} a_{1}, h_{j} a_{2}, \ldots, h_{j} a_{k+1}\right\} \cap H_{i}\right|$ is equal to the number of $H_{i}$ in $\left\{H a_{1}, H a_{2}, \ldots, H a_{k+1}\right\}$, where $i \in\{0,1, \ldots, n\}, j \in\{1,2\}$. Let $a_{i_{1}} H \neq a_{i_{2}} H$ and we consider a family of cosets
$$\{a_{i_1}H, a_{i_2}a_{i_1}H, a_{i_1}a_{i_2}a_{i_1}H,..., \underbrace{a_{i_s}a_{i_{3-s}}... a_{i_1}}_{n+1 \ \textrm{times}}H\}, \ s\in\{1,2\}.$$

Then we can find at least two equal cosets. Hence, there exists $x \in a_{i_{1}}H$ such that $x a_{i_{1}} \in H$. By above facts and (\ref{rb}), we can conclude that if
$$H^{'}, H^{''} \in \left\{yH_i \mid y\in S_1(x), x\in H_i, \ i\in \{0,1,..., n\} \right\},$$
 then there exists $x^{'}, x^{''}\in H_i$ such that $x_{\downarrow}^{'}\in H^{'}$ and $x_{\downarrow}^{''}\in H^{''}$. By using this fact we consider $H$-periodic solutions of (\ref{9}). For any $x_{1}, x_{2}\in H_i, \ i\in \{0, 1,...,n\}$ one gets
\begin{equation}
z_{x_{1}}=\lambda  \prod_{y \in S\left(x_{1}\right)} \frac{1}{1+\left\|z_{y}\right\|}, \quad z_{x_{2}}=\lambda  \prod_{y \in S\left(x_{2}\right)} \frac{1}{1+\left\|z_{y}\right\|}.
\end{equation}
Consequently,
$$1=\frac{z_{x_1}}{z_{x_2}}=\frac{\lambda  \prod_{y \in S\left(x_{1}\right)} \frac{1}{1+\left\|z_{y}\right\|}}{\lambda  \prod_{y \in S\left(x_{2}\right)} \frac{1}{1+\left\|z_{y}\right\|}}=\frac{1+\|z_{x_{1_{\downarrow}}}\|}{1+\|z_{x_{2_{\downarrow}}}\|}.$$
Thus, $\|z_{x_{1_{\downarrow}}}\|=\|z_{x_{2_{\downarrow}}}\|$. Hence, $\left\|z_{x_1}\right\|=\left\|z_{x_2}\right\|=a, \forall x_1, x_2 \in S_{1}(x)$. Analogously, if we denote $z_x:=b, x\in H$ then we obtain that $z_{y^{'}}=b$ for all $y^{'}\in S_{1}(y)$ where $y\in S_1(x)$.
If $a=b$ then $H$-periodic solution is translation-invariant. If $a\neq b$ then $H$-periodic Gibbs measure is $G_{k}^{(2)}$- periodic with period two. \end{proof}

\begin{thm}\label{thm2}\cite{24} Let $k \geq 2$ and $\Lambda_{c r}(k)=\frac{k^{k}}{(k-1)^{k+1}}$. Then for the HC model with a countable number of states (corresponding to the graph $G^{\ast}$), the following statements are true:
\begin{itemize} \item  If $0<\|\lambda\| \leq \Lambda_{cr}$ then there is exactly one $G_{k}^{(2)}-$ periodic Gibbs measure $\mu_{0}$, which is translation invariant;
\item If $\|\lambda\|>\Lambda_{cr}$, then there are exactly three $G_{k}^{(2)}$-periodic Gibbs measures $\mu_{0}, \mu_{1}, \mu_{2}$, where the measures $\mu_{ 1}$ and $\mu_{2}$ are $G_{k}^{(2)}$-periodic (not translation-invariant). \end{itemize}
\end{thm}

From Theorem \ref{th1} and Theorem \ref{thm2} we can conclude the following corollary.
\begin{cor}\label{cor3} Let $k \geq 2$ and $\Lambda_{c r}(k)=\frac{k^{k}}{(k-1)^{k+1}}$. Then for the HC model with a countable number of states (corresponding to the graph $G^{\ast}$), the following statements are true:
\begin{itemize} \item  If $0<\|\lambda\| \leq \Lambda_{cr}$ then for any normal subgroup $H$, each $H$-periodic Gibbs measure is translation-invariant and it is unique;
\item If $\|\lambda\|>\Lambda_{cr}$, then there are only three periodic Gibbs measures and two of them are $G_{k}^{(2)}$-periodic and one translation-invariant Gibbs measures. \end{itemize}
\end{cor}

\section{Bleher-Ganikhodjaev construction}

Let $\Gamma_{0}^{k}=(V^0, L^0)$ be a semi-infinite tree. Note that all above notations in $V$ are the same as the notations in $V^0$.  On the tree $\Gamma_{0}^{k}$ one can introduce a partial ordering, by saying that $y>x$ if there exists a path $x=x_{0}, x_{1}, \ldots, x_{n}=y$ from $x$ to $y$ that ``goes upwards" and the set of vertices $V_{x}^{0}=\left\{y \in V^{0} \mid y \geq x\right\}$ and the edges connecting them from the semi-infinite tree $\Gamma_{x}^{k}$ ``growing" from the vertex $x \in V^{0}$.

We can represent an arbitrary (finite or infinite) path $x^{0}=x_{0}<x_{1}<x_{2}<\ldots$ starting from the point $x^{0}$ by a sequence $i_{1} i_{2} i_{3} \ldots$, where $i_{n}=0,1, \ldots, k-1$ (see \cite{BG} and \cite{5}).

Let $\pi=\left\{x^{0}=x_{0}<x_{1}<\ldots\right\}$ be an infinite path, represented by the sequence $i_{1} i_{2} \ldots$ We assign the real number
$$
t=t(\pi)=\sum_{n=1}^{\infty} \frac{i_{n}}{k^{n}}, \quad 0 \leq t \leq 1
$$
to the path $\pi$. This assignment is a one to one correspondence everywhere except at those numbers $t$ that can be decomposed into a finite sum,
$$
t=\sum_{n=1}^{N} \frac{i_{n}}{k^{n}}, i_{N} \neq 0 .
$$
Denote
$$
Q_{k}=\left\{0<t<1: t=\sum_{n=1}^{N} \frac{i_{n}}{k^{n}}, i_{N} \neq 0\right\} .
$$
If $t \in Q_{k}$, then $t$ is known to be decomposed in two ways:
$$
\begin{gathered}
t=\sum_{n=1}^{N} \frac{i_{n}}{k^{n}}, i_{N} \neq 0, \quad t=\sum_{n=1}^{N-1} \frac{i_{n}}{k^{n}}+\frac{\left(i_{N}-1\right)}{k^{N}}+\sum_{n=N+1}^{\infty} \frac{k-1}{k^{n}} .
\end{gathered}
$$

The first of these sequences is denoted by $\{i_{n}^{(1)}, n=1,2, \ldots\}$ and the second by $\{i_{n}^{(2)}, n=1,2, \ldots\}$.

Let the paths $\pi(t, 1)$ and $\pi(t, 2)$ be represented by the respective sequences $i_{1}^{(1)} i_{2}^{(1)} i_{3}^{(1)} \ldots$ and $i_{1}^{(2)} i_{2}^{(2)} i_{3}^{(2)} \ldots .$
Let $\pi=\left\{x^{0}=x_{0}<x_{1}<\ldots\right\}$ be an infinite path. We assign the set of numbers $z^{\pi}=\left\{z_{x}^{\pi}, x \in V^{0}\right\}$ satisfying equation (\ref{9}) to the path $\pi$. For $x \in W_{n}$, $n=1,2, \ldots$ the set $z^{\pi}$ is unambiguously defined by
\begin{equation}\label{2.50}
z_{x}^{\pi}=\left\{\begin{array}{l}
\alpha^{*}\lambda, \text { if } x \prec x_{n}, x \in W_{n}, \\
\beta^{*}\lambda, \text { if } x_{n} \prec x, x \in W_{n}.
\end{array}\right.
\end{equation}

\begin{thm}\label{THEOREM 2.28.} Let $\Lambda_{cr}<\|\lambda\|<\frac{1}{\beta^{\ast}-\alpha^{\ast}}, \ \lambda\in \mathcal{L}^+$  then for any infinite path $\pi$, there exists a unique set of numbers $z^{\pi}=$ $\left\{z_{x}^{\pi}, x \in V^{0}\right\}$ satisfying equations (\ref{9}) and (\ref{2.50}). \end{thm}
\begin{proof} On $W_{n}$, we define the set
$$
z_{x}^{(n)}=\left\{\begin{array}{l}
\alpha^{\ast}\lambda, \text { if } x \prec x_{n}, x \in W_{n}, \\
\beta^{\ast}\lambda, \text { if } x_{n} \prec x, x \in W_{n}, \\
z_{x}^{(n)}, \text { if } x=x_{n},
\end{array}\right.
$$
where $z_{x_{n}}^{(n)} \in\left[\alpha^{\ast}\lambda, \beta^{\ast}\lambda\right]$ is an arbitrary number. We extend the definition of $z_{x}^{(n)}$ for all $x \in V_{n}=\cup_{m=0}^{n} W_{m}$ using recursion equations (\ref{9}). Put $z_{x}^{(n)}=\alpha_x^{(n)}\lambda$ and $z_x=\alpha_x \lambda$ then we now prove that the limit
\begin{equation}\label{2.51}
\alpha_{x}=\lim _{n \rightarrow \infty} \alpha_{x}^{(n)}
\end{equation}
exists for every fixed $x \in V^{0}$ and is independent of the choice of $z_{x}^{(n)}$ for $x=x_{n}$. If $x \in W_{n-1}$ and $x \prec x_{n-1}$, then
$$
z_{x}^{(n)}=\lambda\prod_{y \in W_{n}, y>x} \frac{1}{1+\beta^{\ast}\|\lambda\|}=\lambda \left(\frac{1}{1+\beta^{\ast}\|\lambda\|}\right)^k=\alpha^{\ast}\lambda .
$$
Similarly, for $x \in W_{n-1}$ and $x_{n-1} \prec x$, we get $z_{x}^{(n)}=\beta^{\ast}\lambda$. Consequently, for any $x \in W_{m}, m \leq n$ we have
\begin{equation}\label{2.53}
z_{x}^{(n)}=\left\{\begin{array}{l}
\alpha^{\ast}\lambda, \text { if } x \prec x_{m}, \ x \in W_{m},\\
\beta^{\ast}\lambda, \text { if } x_{m} \prec x, \ x \in W_{m}.
\end{array}\right.
\end{equation}
This implies that limit (\ref{2.51}) exists for $x \in W_{m}$ and $x \neq x_{m}$ and
$$
z_{x}=\left\{\begin{array}{l}
\alpha^{\ast}\lambda, \text { if } x \prec x_{m}, \ x \in W_{m}, \\
\beta^{\ast}\lambda, \text { if } x_{m} \prec x, \ x \in W_{m} .
\end{array}\right.
$$
Therefore, we only need to establish that limit (\ref{2.51}) exists for $x=x_{m}$. Let $1 \leq l \leq n$. Then
$$
z_{x_{l-1}}^{(n)}=\lambda\prod_{y \in S(x_{l-1})} \frac{1}{1+\|z_{y}^{(n)}\|}.
$$
Consider two sets $\left\{\bar{z}_{x}^{(n)}:=\bar{\alpha}_{x}^{(n)}\lambda, \ x \in V_{n}\right\}$ and $\left\{\tilde{z}_{x}^{(n)}:=\tilde{\alpha}_{x}^{(n)}\lambda,\ x \in V_{n}\right\}$ which correspond to two values $\bar{z}_{x}^{(n)}$ and $\tilde{z}_{x}^{(n)}$ for $x=x_{n}$, in (\ref{2.50}), then from (\ref{2.53}) we get
\begin{equation}\label{2.54}
\|\tilde{z}_{x_{l-1}}^{(n)}-\bar{z}_{x_{l-1}}^{(n)}\|=\|\lambda\||\bar{\alpha}_{x_{l-1}}^{(n)}-\tilde{\alpha}_{x_{l-1}}^{(n)}|=
\|\lambda\|\prod_{y\in S(x_{l-1})}\left|\frac{1}{1+\bar{\alpha}_{y}^{(n)}\|\lambda\|}-\frac{1}{1+\tilde{\alpha}_{y}^{(n)}\|\lambda\|}\right|.
\end{equation}
Since $\tilde{z}_{y}^{(n)}=\bar{z}_{y}^{(n)}$ for any $y \neq x_{l}, y \in W_{l}$ one gets
$$\|\tilde{z}_{x_{l-1}}^{(n)}-\bar{z}_{x_{l-1}}^{(n)}\|=\|\lambda\|\prod_{y\in S(x_{l-1})\setminus \{x_l\}}\frac{1}{1+\bar{\alpha}_{y}^{(n)}\|\lambda\|}
\left|\frac{1}{1+\bar{\alpha}_{x_l}^{(n)}\|\lambda\|}-\frac{1}{1+\tilde{\alpha}_{x_l}^{(n)}\|\lambda\|}\right|.$$
Since $\alpha^{\ast}\leq\alpha_x\leq\beta^{\ast}, \ x\in V$ we have
$$\|\tilde{z}_{x_{l-1}}^{(n)}-\bar{z}_{x_{l-1}}^{(n)}\|\leq \frac{\|\lambda\|}{(1+\alpha^{\ast}\|\lambda\|)^{k+1}}\|\tilde{z}_{x_{l}}^{(n)}-\bar{z}_{x_{l}}^{(n)}\|=
\frac{\|\lambda\|\beta^{\ast}}{(1+\alpha^{\ast}\|\lambda\|)}\|\tilde{z}_{x_{l}}^{(n)}-\bar{z}_{x_{l}}^{(n)}\|.$$
Let $\theta:=\frac{\|\lambda\|\beta^{\ast}}{(1+\alpha^{\ast}\|\lambda\|)}$ then since $\|\lambda\|<\frac{1}{\beta^{\ast}-\alpha^{\ast}}$ we get $\theta<1$.

Hence $$\left\|\tilde{z}_{x_{l-1}}^{(n)}-\bar{z}_{x_{l-1}}^{(n)}\right\| \leq \theta\left\|\tilde{z}_{x_{l}}^{(n)}-\bar{z}_{x_{l}}^{(n)}\right\|.$$
Iterating this inequality we obtain
\begin{equation}\label{2.57} \quad\left|\tilde{h}_{x_{m}}^{(n)}-\bar{h}_{x_{m}}^{(n)}\right| \leq \theta^{n-m}\left|\tilde{h}_{x_{n}}^{(n)}-\bar{h}_{x_{n}}^{(n)}\right|.\end{equation}
For arbitrary $N, M>n$, we now consider the sets $\left\{z_{x}^{(N)}, x \in V_{N}\right\}$ and $\left\{z_{x}^{(M)}, x \in V_{M}\right\}$ determined by initial conditions of form (\ref{2.50}) for $x \in W_{N}$ and $x \in W_{M}$ respectively and by recursion equations (\ref{9}). We set $\bar{z}_{x_{n}}^{(n)}=z_{x_{n}}^{(N)}, \tilde{z}_{x_{n}}^{(n)}=z_{x_{n}}^{(M)}$. Then inequalities (\ref{2.57}) imply
$$
\left|z_{x_{m}}^{(N)}-z_{x_{m}}^{(M)}\right| \leq \theta^{n-m}\left|z_{x_{n}}^{(N)}-z_{x_{n}}^{(M)}\right| \leq 2 \beta^{*} \theta^{n-m}.
$$
This estimate implies that the sequence $z_{x_{m}}^{(n)}$ satisfies the Cauchy criterion as $n \rightarrow \infty$ for a fixed $m$; therefore, limit (\ref{2.51}) exists and is independent of the choice of $z_{x_{n}}^{(n)}$ in (\ref{2.50}). Because, by construction, the sets $\{z_{x}^{(n)}\}$ satisfy equation (\ref{9}) before taking the limit, so does $\left\{z_{x}\right\}$. The uniqueness of $\left\{z_{x}\right\}$ obviously follows from estimate (\ref{2.57}).\end{proof}

\begin{lemma}\label{2.29} For every $t \in Q_{k}, t=\sum_{n=1}^{N} \frac{i_{n}}{k^{n}}$, the number sets $z^{\pi(t, 1)}$ and $z^{\pi(t, 2)}$ are identical, and $z_{x}^{\pi(t, 2)}=\beta^{*}$ if $x=x_{n}, n \geq N+1$. \end{lemma}

\begin{proof} For $n>N$, we set
$$
z_{x}=\left\{\begin{array}{l}
\alpha^{\ast}, \text { if } x \prec x_{n}, \\
\beta^{\ast}, \text { if } x_{m} \preceq x,
\end{array}\right.
$$
and extend the definition of the set $\left\{z_{x}, x \in V^{0}\right\}$ for $x \in V_{N}$, using recursion equations (\ref{9}). Then, it is easy to see that the set $\left\{z_{x}, x \in V^{0}\right\}$ satisfies these equations for all $x \in V^{0}$; further, from Theorem \ref{thm4}, one gets the constructed set $\left\{z_{x}, x \in V^{0}\right\}$ is coincide with $\left\{z_{x}^{\pi(t, 2)}, x \in V^{0}\right\}$. Analogously, the set $\left\{z_{x}, x \in V^{0}\right\}$ is also identical to $\left\{z_{x}^{\pi(t, 2)}, x \in V^{0}\right\}$; therefore, $z^{\pi(t, 1)}=z^{\pi(t, 2)}$. \end{proof}

It follows from Lemma \ref{2.29} that for any point $t \in[0,1]$, the number set $z^{\pi(t)}=$ $\left\{z_{x}^{\pi(t)}, x \in V^{0}\right\}$ is unambiguously defined. Let $z_{0}(t)=z_{x^{0}}^{\pi(t)}$,
where $x^{0}$ is the root vertex of the graph $\Gamma_{0}^{k}$.

\begin{lemma}\label{2.30} The function $z_{0}(t), t \in[0,1]$ is a strictly decreasing continuous function, and $z_{0}(0)=\beta^{*}, z_{0}(1)=\alpha_{*}$. \end{lemma}
\begin{proof} Let
$$
t=\sum_{n=1}^{\infty} \frac{j_{n}}{k^{n}}, \quad s=\sum_{n=1}^{\infty} \frac{i_{n}}{k^{n}}, t>s
$$
Then there exists $N$ such that

$$
k^{N} \sum_{n=1}^{N} \frac{j_{n}}{k^{n}}-k^{N} \sum_{n=1}^{N} \frac{i_{n}}{k^{n}} \geq 2
$$
(otherwise $t=s$ ).

Let $\pi(t)=\left\{x^{0}=x_{0}<x_{1}<x_{2}<\ldots\right\}$ and $\pi(s)=\left\{x^{0}=y_{0}<y_{1}<y_{2}<\ldots\right\}$. Then $y_{N} \prec x_{N}$ and at least one more vertex $z_{N}$ exists between $x_{N}$ and $y_{N}$, i.e., $y_{N} \prec z_{N} \prec x_{N}$. By conditions (\ref{2.50}), this leads to

$$
z_{x}^{\pi(t)} \begin{cases}\leq z_{x}^{\pi(s)}, & x \in W_{n}, \\ <z_{x}^{\pi(s)}, & x=z_{n} .\end{cases}
$$
Setting $n=0$, we obtain $z_{x^{0}}^{\pi(t)}<z_{x^{0}}^{\pi(s)}$, which proves that the function $z_{0}(t)$ is strictly monotonic. We now show it to be continuous. Let
$$
t=\sum_{n=1}^{\infty} \frac{j_{n}}{k^{n}}, s=\sum_{n=1}^{\infty} \frac{i_{n}}{k^{n}},
$$
with
$$
i_{n}=j_{n} \text {, if } n \leq N
$$
Then paths $\pi(t), \pi(s)$ coincide up to the level $N: x_{n}=y_{n}, n \leq N$. We set $\bar{z}_{x_{N}}^{(N)}=$ $z_{x_{N}}^{\pi(t)}, \tilde{z}_{x_{N}}^{(N)}=z_{x_{N}}^{\pi(s)}$. Then, by (\ref{2.57}) with $n=N$ and $m=0$, we get
$$
\left\|z_{0}(t)-z_{0}(s)\right\| \leq \theta^{N} 2 \beta^{*} .
$$
Then it is obvious that
$$
\left\|z_{0}(t)-z_{0}(s)\right\| \leq \theta^{N} 4 \beta^{*}, \text { if }|t-s| \leq k^{-N-1} .
$$
This estimate leads to
$$
\left\|z_{0}(t)-z_{0}(s)\right\| \leq C|t-s|^{\alpha},
$$
where $\alpha=-\frac{\ln \theta}{\ln k}$, which proves the continuity of $z_{0}(t)$. Equalities $z_{0}(0)=\beta^{*}, z_{0}(1)=\alpha^{*}$ follow from Lemma \ref{2.29}.\end{proof}

Hence (similarly to  Bleher-Ganikhodjaev construction), we obtain uncountable number of solutions to (\ref{9}).  Lemma \ref{2.30} implies that the number sets $z^{\pi(t)}$ are distinct for different $t \in[0,1]$. By Proposition \ref{lemma new}, $z^{\pi(t)}$ is a normalisable and by Theorem \ref{tz} there is a one-to-one correspondence between normalisable
boundary laws and Gibbs measures. Thus, we can correspond Gibbs measure $\mu^{t}$ for any solution $z^{\pi(t)}$ to (\ref{9}). Hence, we can conclude the following theorem.

\begin{thm}\label{final} Let $\Lambda_{cr}<\|\lambda\|<\frac{1}{\beta^{\ast}-\alpha^{\ast}}$. Then for any $t\in [0,1]$ there is a Gibbs measure $\mu_t$ corresponding to $z^{\pi(t)}$ such that $\mu_t\neq \mu_s$, $t\neq s$, $s,t\in[0,1]$.
\end{thm}
\section*{Acknowledgements}
The work supported by the fundamental project (number: F-FA-2021-425)  of The Ministry of Innovative Development of the Republic of Uzbekistan.

\section*{Statements and Declarations}

{\bf	Conflict of interest statement:}
On behalf of all authors, the corresponding author (F. Haydarov) states that there is no conflict of interest.

\section*{Data availability statements}
The datasets generated during and/or analysed during the current study are available from the corresponding author on reasonable request.


\begin{thebibliography}{99}

\bibitem{12} Brightwell G., H\"{a}ggstr\"{o}m O., Winkles P.: Non monotonic behavior in hard-core and WidomRowlinson models, {\it Jour. Stat. Phys.} {\bf 94} (1999), pp. 415-435.

\bibitem{16} Brightwell G., Winkler P.: Graph homomorphisms and phase transitions,  {\it J. Combin. Theory Ser.B.} {\bf 77}, (1999) pp. 221-262.

 \bibitem{BG}  Bleher P. M.,  Ganikhodjaev N. N.: On pure phases of the Ising model on the Bethe lattice, {\it Theor.
Probab. Appl.} {\bf 35} (1990), 216-227.

\bibitem{18} Bogachev L. V., Rozikov U.A.: On the uniqueness of Gibbs measure in the Potts model on a Cayley tree with external field,  {\it J.Stat. Mech.} (2019),  073205.


\bibitem{6}  Ganikhodjaev N. N.: Limiting Gibbs measures of Potts model with countable set of spin values,  {\it J. Math. Anal. Appl.} {\bf 336} (2007), 693-703.

\bibitem{a} Henning F., K\"{u}lske C, Le NyA,   and Rozikov U.A. : Gradient Gibbs measures for the SOS model with countable values on a Cayley tree,  {\em Electron. J. Probab.} \textbf{24} (2019), Paper No. 104, 23 pp.

  \bibitem{b} Henning F., K\"{u}lske C: Existence of gradient Gibbs measures on regular trees which are not translation invariant,  arXiv:2102.11899 [math.PR]


\bibitem{9} Henning F., K\"{u}lske C.: Coexistence of localized Gibbs measures and delocalized gradient Gibbe measures on trees,  {\it Ann. Appl. Probab.} {\bf 31(5)}, (2021), pp. 2284-2310.

\bibitem{georgii} Georgii H.O.: Gibbs Measures and Phase Transitions,  {\it de Gruyter Studies in Mathematics} {\bf 9}, (2011).

\bibitem{13} Kelly F.P.: Stochastic models of computer communication systems, {\it J. Roy. Stat. Soc. Ser. B} {\bf 47} (1985), pp. 379-395.

\bibitem{c}  K\"ulske C and Schriever P.: Gradient Gibbs measures and fuzzy transformations on trees, {\em Markov Process. Relat. Fields}, \textbf{23}, (2017), 553-590.

\bibitem{15} Khakimov R.M., Makhammadaliev M.T.: Uniqueness and nonuniqueness conditions for weakly periodic Gibbs measures for the Hard-Core model. {\it Theor. Math. Phys.} {\bf 204(2)}, pp. 1059-1078 (2020).

\bibitem{14}  Mazel A. E., Suhov Yu.: Random surfaces with two-sided constraints: an application of the theory of dominant ground states {\it J. Statist. Phys.} {\bf 64}, pp. 111-134 (1991).

\bibitem{2} Preston C. J. Gibbs States on countable sets.  {\it Cambridge Tracts Math.}  (1974).

\bibitem{5}  Rozikov U.A.: Gibbs measures on a Cayley tree, {\it World Sci. Pub, Singapore} (2013).

\bibitem{24} Rozikov U.A., Khakimov R.M., Makhammadaliev M.T.:  Gibbs measures for a HC model with a countable number of states on a Cayley tree.  {\it arXiv:2205.02025}.

\bibitem{Rt}  Rozikov U. A. Structures of partitions of the group representation of the Cayley tree into cosets by finite-index normal subgroups, and their applications to the description of periodic Gibbs distributions. {\it Theoret. and Math. Phys.} {\bf 112}(1) (1997), 929-933.

\bibitem{Zach} Zachary S.: Countable state space Markov random fields and Markov chains on trees, {\it Ann. Probab.} {\bf 11(4)}, pp. 894-903 (1983).

\end{thebibliography}
\end{document}